\documentclass[11pt,oneside]{amsart}

\usepackage{amsmath,amsthm,verbatim,amssymb,amsfonts,amscd, graphicx}
\usepackage{mathrsfs}
\usepackage[utf8]{inputenc}
\usepackage[retainorgcmds]{IEEEtrantools}
\usepackage{tikz-cd}
\usepackage{graphics}
\usepackage[all]{xy}
\usepackage{incgraph}
\usepackage{mathtools}
\usepackage{graphicx}
\usepackage{marginnote}
\usepackage{comment}
\usepackage{tikz}
\usetikzlibrary{arrows.meta}
\usetikzlibrary{bending}

\usepackage{hyperref}

\theoremstyle{plain}
\newtheorem{theorem}{Theorem}[section]
\newtheorem{corollary}[theorem]{Corollary}
\newtheorem{lemma}[theorem]{Lemma}
\newtheorem{proposition}[theorem]{Proposition}
\newtheorem{conjecture}[theorem]{Conjecture}

\newtheorem{definition}[theorem]{Definition}

\theoremstyle{remark}
\newtheorem{remark}[theorem]{Remark}

\def\E{\mathcal E} 

\def\makeop#1{\expandafter\def\csname#1\endcsname
  {\mathop{\rm #1}\nolimits}\ignorespaces}
\makeop{Hom}   \makeop{End}   \makeop{Aut}   \makeop{Isom}  \makeop{Pic}
\makeop{Gal}   \makeop{ord}   \makeop{Char}  \makeop{Div}   \makeop{Lie}
\makeop{PGL}   \makeop{Corr}  \makeop{PSL}   \makeop{sgn}   \makeop{Spf}
\makeop{Spec}  \makeop{Map}    \makeop{Nm}    \makeop{Fr}    \makeop{disc}
\makeop{Proj}  \makeop{supp}  \makeop{ker}   \makeop{im}    \makeop{dom}
\makeop{coker} \makeop{Stab}  \makeop{SO}    \makeop{SL}    \makeop{SL}
\makeop{Cl}    \makeop{cond}  \makeop{Br}    \makeop{inv}   \makeop{rank}
\makeop{id}    \makeop{Fil}   \makeop{Frac}  \makeop{GL}    \makeop{SU}
\makeop{Trd}   \makeop{Sp}    \makeop{Tr}    \makeop{Trd}   \makeop{diag}
\makeop{Res}   \makeop{ind}   \makeop{depth} \makeop{Tr}    \makeop{st}
\makeop{Ad}    \makeop{Int}   \makeop{tr}    \makeop{Sym}   \makeop{can}
\makeop{length}\makeop{SO}    \makeop{torsion} \makeop{GSp} \makeop{Ker}
\makeop{Adm}   \makeop{Frob}  \makeop{id}    \makeop{Tor}   \makeop{Ind}
\makeop{CoInd} \makeop{Inf}   \makeop{Cor}   \makeop{CoInf} \makeop{coLie}
\makeop{Ann}   \makeop{Shv}   \makeop{perf}  \makeop{Cont}  \makeop{Perfd}
\makeop{can}

\newcommand{\Z}{\mathbb Z}
\newcommand{\Q}{\mathbb Q}

\newcommand{\C}{\mathbb C}
\newcommand{\N}{\mathbb N}
\renewcommand{\O}{\mathcal O} 

\newcommand{\Ainf}{A_{\mathrm{inf}}} 

\newcommand{\RepZp}{\mathrm{Rep}_{\Z_p}}

\newcommand{\tbfA}{\tilde{\mathbf{A}}}
 
\usepackage{relsize}
\usepackage[bbgreekl]{mathbbol}
\usepackage{amsfonts}
\DeclareSymbolFontAlphabet{\mathbb}{AMSb} 
\DeclareSymbolFontAlphabet{\mathbbl}{bbold}
\newcommand{\Prism}{{\mathlarger{\mathbbl{\Delta}}}}

\makeatletter
\newcommand{\colim@}[2]{%
  \vtop{\m@th\ialign{##\cr
    \hfil$#1\operator@font colim$\hfil\cr
    \noalign{\nointerlineskip\kern1.5\ex@}#2\cr
    \noalign{\nointerlineskip\kern-\ex@}\cr}}%
}
\newcommand{\colim}{%
  \mathop{\mathpalette\colim@{\rightarrowfill@\textstyle}}\nmlimits@
}
\makeatother

\makeatletter
\newcommand*{\da@rightarrow}{\mathchar"0\hexnumber@\symAMSa 4B }
\newcommand*{\da@leftarrow}{\mathchar"0\hexnumber@\symAMSa 4C }
\newcommand*{\xdashrightarrow}[2][]{%
  \mathrel{%
    \mathpalette{\da@xarrow{#1}{#2}{}\da@rightarrow{\,}{}}{}%
  }%
}
\newcommand{\xdashleftarrow}[2][]{%
  \mathrel{%
    \mathpalette{\da@xarrow{#1}{#2}\da@leftarrow{}{}{\,}}{}%
  }%
}
\newcommand*{\da@xarrow}[7]{%
  \sbox0{$\ifx#7\scriptstyle\scriptscriptstyle\else\scriptstyle\fi#5#1#6\m@th$}%
  \sbox2{$\ifx#7\scriptstyle\scriptscriptstyle\else\scriptstyle\fi#5#2#6\m@th$}%
  \sbox4{$#7\dabar@\m@th$}%
  \dimen@=\wd0 %
  \ifdim\wd2 >\dimen@
    \dimen@=\wd2 %
26

  \fi
  \count@=2 %
  \def\da@bars{\dabar@\dabar@}%
  \@whiledim\count@\wd4<\dimen@\do{%
    \advance\count@\@ne
    \expandafter\def\expandafter\da@bars\expandafter{%
      \da@bars
      \dabar@ 
    }%
  }%
  \mathrel{#3}%
  \mathrel{%
26

    \mathop{\da@bars}\limits
    \ifx\\#1\\%
    \else
      _{\copy0}%
    \fi
    \ifx\\#2\\%
    \else
      ^{\copy2}%
    \fi
  }%
  \mathrel{#4}%
}
\makeatother

\begin{document}
\author{Heng Du}
\address[Heng Du]{Yau Mathematical Sciences Center, Tsinghua University, Beijing 100084, China}
\email{hengdu@mail.tsinghua.edu.cn}

\author{Tong Liu}
\address[Tong Liu]{Department of Mathematics, Purdue University, 150 N. University Street, West Lafayette, Indiana 47907, USA}
\email{tongliu@math.purdue.edu}

\title[A new method for overconvergence of $(\varphi, \Gamma)$-modules]{A new method for overconvergence of $(\varphi, \Gamma)$-modules}
\maketitle

\begin{abstract}
We show all Laurent $F$-crystals over $p$-adic fields are overconvergent. 
\end{abstract}
\tableofcontents

\section{Introduction}
One aim in the study of $p$-adic Hodge theory is to relate good categories of $p$-adic Galois representations of $\mathrm{Gal}_{\Q_p}$ with certain categories of semilinear algebraic objects over some big rings. Since it is usually easier to work with those semilinear algebraic objects when studying Galois representations in general coefficients, and when constructing or deforming Galois representations. When we want to consider all $p$-adic Galois representations of $\mathrm{Gal}_{\Q_p}$ over $\Z_p$-lattices, then the semilinear algebra category one should use is the category of \'etale $(\varphi,\Gamma)$-modules over $\mathbf{A_{\Q_p}}$. Here the latter category is a certain linear algebra category over the ring 
$$
\mathbf{A}_{\Q_p}\coloneqq\{f=\sum_{i\in\Z}a_iX^i\mid a_i\in \Z_p \text{ and }\lim_{i\to -\infty}a_i=0\}
$$
equipped with a semilinear action $\varphi$ and $\Z_p^\times$. When $K$ is a $p$-adic field, here in this paper we will refer to the case that $K$ being a finite extension of $\Q_p$, the above equivalence between $p$-adic Galois representations and \'etale $(\varphi,\Gamma)$-modules are generalized using the field of norms construction of Fontaine-Wintenberger\cite{FontaineWintenberger-fieldofnorms}. It turns out that \'etale $(\varphi,\Gamma)$-modules over $\mathbf{A}_{\Q_p}$ always descent to so-called overconvergent \'etale $(\varphi,\Gamma)$-modules and this was first established by Cherbonnier-Colmez in \cite{Cherbonnier-Colmez}. In the case when $K=\Q_p$, their result implies every \'etale $(\varphi,\Gamma)$-modules over $\mathbf{A}_{\Q_p}$ descent to the Robba ring $\mathbf{A}_{\Q_p}^\dagger$, the subring consisting of functions $f$ that converges on some annulus $(r, 1)$ with $r$ sufficiently closed to $1$.

There are two paralleled directions of generalizing the theory of \'etale $(\varphi,\Gamma)$-modules and the corresponding overconvergent theories, to a family version or to a relative version. To be more precise, the family version usually refers to Galois representations in general coefficients, cf. \cite{Dee}, \cite{BergerColmez-Families}, and recently \cite{Porat-Overconvergent}. On the other hand, the relative version refers to replacing Galois groups with \'etale fundamental groups of some geometric objects over the base $p$-adic field, cf. \cite{AB-overconvergence} and \cite{KedlayaLiu-relativeII}. Other than the works we mentioned, people also ask similar questions to $(\varphi,\tau)$-modules developed in \cite{Caruso-phitau}, and the overconvergence of $(\varphi,\tau)$-modules were obtained in \cite{Gao-Liu-Loosecrystallinelifts} and \cite{Gao-Poyeton}. Moreover, there is the theory about overconvergence for Lubin-Tate tower which is considered by Berger in \cite{Berger-multilocanalytic}.

In \cite{KedlayaLiu-relativeII}, Kedlaya-Liu tried to give an axiomatic approach to the theory of overconvergence of $(\varphi,\Gamma)$-modules, and their idea is that overconvergence is closely related to a property of a perfectoid finite \'etale tower which they called decompleteness. Using such an idea, they show the decompleteness of the cyclotomic tower, which recovers the result of Cherbonnier-Colmez. Their formalism also succeeds in dealing with relative situations that show relative toric towers are decompleting, and this result plays an important role in the construction of the Riemann-Hilbert functor in \cite{Zhu-Liu-RH}.  Kedlaya-Liu expects their formalism to apply to the Kummer tower as well, which will reprove the overconvergence of $(\varphi,\tau)$-modules, cf. Conjecture 6.2.6. in $loc.$ $cit$. However, we will show that the overconvergent period rings defined by Kedlaya-Liu in Kummer's case are too small and very few $(\varphi,\tau)$-modules are overconvergent in their sense, cf. the Appendix~\ref{app}. 

This paper gives a sheaf theoretic approach to the overconvergence of $(\varphi, \Gamma)$-modules. We will use the theory of prisms developed by Bhatt-Scholze in \cite{BS19} and \cite{BS2021Fcrystals}. Let $K$ be a finite extension of $\Q_p$, and let $X$ be a smooth formal scheme over $\O_K$. Consider the absolute prismatic site $X_\Prism$ over $X$, which equipped with the structure sheaf $\O_\Prism$ and ideal sheaf $\mathcal{I}_\Prism$, also define $\O_{\E,\Prism}=\O_\Prism[1/\mathcal{I}_\Prism]^\wedge_p$. In \cite{BS2021Fcrystals}, they define the category of Laurent $F$-crystals consisting of crystals in finite locally free $\O_{\E,\Prism}$-modules together with a Frobenius semilinear endomorphism which is \'etale, i.e., the linearization of the Frobenius semilinear endomorphism is an isomorphism, then they can show

\begin{theorem}[{\cite[Corollary 3.8]{BS2021Fcrystals}}, \cite{wu2021galois}]
The category of Laurent $F$-crystals over $X$ is equivalent to the category of $\Z_p$-local systems on the generic fiber of $X$.
\end{theorem}

In this paper, we will construct a sheaf $\O_{\mathcal{E},\Prism}^{\dagger}$ over $X_{\Prism^\circ}$ which is the subcategory of $X_{\Prism}$ consisting \textit{transversal} prisms over $X$. Then we define the category of \textit{overconvergent} Laurent $F$-crystals to be crystals in finite locally free $\O_{\mathcal{E},\Prism}^{\dagger}$-modules together with a Frobenius semilinear endomorphism which is \'etale in \S\ref{sec:overconvergentsheaf}. 

\begin{theorem}[{Theorem~\ref{thm:mainwithproof}}]\label{thm:mainintro}
Let $K$ be a finite extension of $\Q_p$ with ring of integers $\O_K$, and let $X=\Spf(\O_K)$, we have all Laurent $F$-crystals over $X$ are overconvergent.
\end{theorem}

For the exact meaning of the above theorem, we will define a functor
\begin{equation}\label{eq:DXintro}
D_X: \mathrm{Vect}(X_{\Prism^\circ},\O_{\mathcal{E},\Prism}^{\dagger})^{\varphi=1} \to \mathrm{Vect}(X_{\Prism},\O_{\mathcal{E},\Prism})^{\varphi=1}.
\end{equation}
in \S\ref{sec:overconvergentsheaf}, where $\mathrm{Vect}(X_{\Prism^\circ},\O_{\mathcal{E},\Prism}^{\dagger})^{\varphi=1} $ and $ \mathrm{Vect}(X_{\Prism},\O_{\mathcal{E},\Prism})^{\varphi=1}$ are the category of overconvergent Laurent $F$-crystals and Laurent $F$-crystals respectively. And Theorem~\ref{thm:mainintro} states that $D_X$ is an equivalence of categories when $X$ is the formal spectrum of the ring of integers of a $p$-adic field. So in particular, by the result of Bhatt-Scholze, $\mathrm{Vect}(X_{\Prism^\circ},\O_{\mathcal{E},\Prism}^{\dagger})^{\varphi=1}$ is also equivalent to the category of $G_K$-stable lattice in $p$-adic representations of $G_K$.

We expect the above theorem to give a uniform approach to all overconvergent theories over $p$-adic fields. In particular, it recovers all the main results in \cite{Cherbonnier-Colmez}, \cite{Gao-Liu-Loosecrystallinelifts} and \cite{Gao-Poyeton} in the case of $p$-adic fields. Here we want to mention the above results can also be treated uniformly using the framework of Berger-Colmez on locally analytic vectors \cite{BergerColmez-Sen} and \cite{Berger-multilocanalytic}, by showing there are ``enough" locally analytic vectors in certain perfect overconvergent period rings. When the tower is not a $p$-adic Lie extension, Berger-Colmez's method doesn't apply, and there is no general framework for proving overconvergent results. Meanwhile, Theorem~\ref{thm:mainintro} can produce many results in the non-Lie extension direction. Theorem~\ref{thm:mainintro} In \S\ref{sec:phiiterate}, we will formulate and prove the ``decompleting'' of any $\Phi$-iterate tower over $K$. Moreover, we can show for a fixed $\Z_p$-Galois representation $T$, the convergent radii are uniformly bounded for all choices of $\Phi$-iterate towers.

\medskip
\noindent
\textbf{Why using the prismatic site.} Prisms are considered as the deperfection of perfectoid rings. We know the overconvergent theory is relatively easy over perfect period rings. Actually, the underlying \'etale $\varphi$-modules already give an equivalence of categories in many cases, cf. \cite{KedlayaLiu-relativeII}. And we can show \eqref{eq:DXintro} fits into the following diagram:
$$
\begin{tikzcd}
\mathrm{Vect}(X_{\Prism^\circ},\O_{\mathcal{E},\Prism}^{\dagger})^{\varphi=1} \arrow[r,"D_X"]\arrow[d] & \mathrm{Vect}(X_{\Prism},\O_{\mathcal{E},\Prism})^{\varphi=1} \arrow[d]\\
\mathrm{Vect}(X_{\Prism^\circ}^{\perf},\O_{\mathcal{E},\Prism}^{\dagger})^{\varphi=1} \arrow[r,"D_X^{\perf}"]& \mathrm{Vect}(X_\Prism^{\perf},\O_{\mathcal{E},\Prism})^{\varphi=1},
\end{tikzcd}
$$
where the vertical arrows are induced by the restrictions to perfect sites. And $D_X^{\perf}$ is the corresponding functor over perfect sites. We will see $D_X^{\perf}$ recovers the overconvergent theory over perfect period rings in \S\ref{sec:perfectdecomp}. And our theory is indeed a deperfection of this result. 

\medskip
\noindent
\textbf{Idea of proof.} We will define the overconvergent period sheaves $\O_{\mathcal{E},\Prism}^{\dagger,n}$ and $\O_{\mathcal{E},\Prism}^{\dagger}$ over $X_{\Prism^\circ}$ in \S\ref{sec:overconvergentsheaf}, to understand sections of these sheaves, we developed a ``Noetherian approximation" method in \S\ref{sec:BKprism}. Then we can define the category of overconvergent Laurent $F$-crystals and we can show that overconvergent Laurent $F$-crystals can be understood using overconvergent \'etale $\varphi$-modules with ``overconvergent" descent isomorphisms in \S\ref{sec:overconvergentsheaf}. Given a $\Z_p$-Galois representation, we will get an overconvergent \'etale $\varphi$-module from the classical overconvergent results, then the question is how to find the overconvergent descent isomorphism. We will use the idea in \cite{Du-Liu-phiGhatmodules}. Briefly speaking, our slogan is that the descent isomorphism should come from the $\tau$-action, where $\tau$ is a specific Galois group element that comes from the $(\varphi,\tau)$-module theory of Caruso in \cite{Caruso-phitau}. To apply our strategy, we will also use some computations made in \cite{Gao-Liu-Loosecrystallinelifts}, where they can estimate the overconvergent radius of $\tau$ using (loose) crystalline lifting tricks together with the theory of $(\varphi,\hat{G})$-modules. However, in \cite{Gao-Liu-Loosecrystallinelifts} they can only get the overconvergent $\tau$-action in some perfect overconvergent period rings, the main obstruction of to ``deperfect" the $\tau$-action is because in $(\varphi,\hat{G})$-modules, one can only see the $\tau$-action after a Frobenius twist. But we will see if we replace the theory of $(\varphi,\hat{G})$-module with the theory of prismatic $(\varphi,\hat{G})$-modules developed in \cite{Du-Liu-phiGhatmodules}, this problem disappears and $\tau$-action provides the descent isomorphism. 

\medskip
\noindent
\textbf{Notations and conventions.} In this paper, $K$ is $p$-adic field, i.e., a finite extension of $\Q_p$. We let $\O_K$ be the ring of integers of $K$ which contains a fixed uniformizer $\varpi$, and $k$ the residue field of $K$. We will write $W=W(k)$, $K_0=W[1/p]$. We fix an algebraic closure $\overline{K}$ over $K$ and write $C$ as the $p$-adic completion of $\overline{K}$. The ring of integers in $C$ will be denoted by $\O_C$. Let $C^\flat$ (resp. $\O_C^\flat$) to be the tilt of $C$ (resp. $\O_C$). Define $\Ainf=\Ainf(\O_C)$. We will fix a compatible system $\{\varpi_n\}_{n\geq 0}$ of $p^n$-th roots with $\varpi_0=\varpi$ as well as a compatible system $\{\zeta_n\}_{n\geq 0}$ of $p^n$-th roots of $1$. Write $\underline{\varpi}^\flat  \coloneqq (\varpi_n)_{n\geq 0}$ and $\underline\zeta^\flat \coloneqq (\zeta_n)_{n\geq 0}$ as elements in $\O_{C}^\flat$. Once the above is given, we can define a map $W[\![u]\!] \to \Ainf$ given by $u \to [\underline{\varpi}^\flat]$.

Let $K _\infty =\bigcup_{n = 1}^\infty K (\varpi _n )$. Then we have a closed subgroup $G_\infty\coloneqq {\rm Gal } (\overline  K / K_\infty)$ of $G_K\coloneqq {\rm Gal } (\overline K  / K)$. Let $L = \cup_{n=1}^\infty K_{\infty} (\zeta_{p ^n})$ and $K_{1^\infty}= \bigcup_{n =1}^\infty K (\zeta_{p ^n})$, then $H_K \coloneqq \Gal (L / K _\infty)$ is a closed subgroup of $\hat G \coloneqq \Gal(L /K) $. Let $\tau$ be a topological generator of $\Gal (L/ K_{1 ^\infty}) $. We will also use $\tau$ to denote a lifting of $\tau$ in $G_K$ when it is acting on an element fixed by $\Gal(\overline{K}/L)$. 

In this paper, for a prism $(A,I)$, if $I$ is generated by $d$, then we will write $(A,d)$ instead of $(A,(d))$ to simplify our symbols. 

\medskip
\noindent
\textbf{Acknowledgments.} It is our pleasure to thank Hansheng Diao, Hui Gao, Yu Min, Yong Suk Moon, Koji Shimizu, Shanwen Wang, and Yupeng Wang for their valuable discussions and interest during the preparation of this paper. The appendix of this paper is based on an email correspondence between the first-named author and Kiran Kedlaya and Ruochuan Liu, H.D. wants to thank them for the discussions on \cite{KedlayaLiu-relativeII}. 

\section{Overconvergent Laurent $F$-crystals}\label{sec:overconvergentsheaf}

Let $X$ be a bounded $p$-adic formal scheme, and let $X_{\Prism}$ be the absolute prismatic site over $X$. We will denote $\O_{\mathcal{E},\Prism}$ to be the sheaf $\O_\Prism[\mathcal{I}_\Prism^{-1}]^\wedge_p$, and we use $\mathrm{Vect}(X_{\Prism},\O_{\E,\Prism})^{\varphi=1}$ to denote the category of Laurent $F$-crystals over $X$.

Let $X_{\Prism^\circ}$ be the subcategory of $X_{\Prism}$ consisting only \textit{transversal} prisms over $X$. Recall that a prism is called \textit{transversal} if $A/I$ is $p$-torsion free. We recall some basic properties of transversal prisms.

\begin{proposition}\label{prop:transversal}
    \begin{enumerate}
        \item If $f\colon(A, I) \to (B, J)$ is a flat morphism of prisms. If $(A, I)$ is
transversal, then so is $(B, J)$. The converse holds if $f$ is a covering;
        \item If $(A,d)$ is a transversal prism, then for all $s\neq r>0$ in $\N$, the sequences $(p,\varphi^s(d))$ and $(\varphi^r(d),\varphi^s(d))$ are regular. 
    \end{enumerate}
\end{proposition}
\begin{proof}
(1) is \cite[Remark 2.4.4]{Bhatt-Lurie-Absoluteprismaticcohomology} and (2) is \cite[Lemma 3.3]{ALB-cyclotomictrace}
\end{proof}

\begin{lemma}\label{lem:transversalsite}
$X_{\Prism^\circ}$ is a site.   
\end{lemma}
\begin{proof}
We need to check that the pushout of a cover along an arbitrary map is again inside $X_{\Prism}^\circ$. For this, we have the pushout of a cover along an arbitrary map is again a covering in $X_{\Prism}$, by Proposition~\ref{prop:transversal}, we have the pushout is transversal. 
\end{proof}

\begin{lemma}
Let $X$ be a $p$-torsion free quasi-syntomic $p$-adic formal scheme. Then the following categories are equivalent:
$$
\mathrm{Vect}(X_{\Prism},\O_{\mathcal{E},\Prism}) \simeq \mathrm{Vect}(X_{\Prism^\circ},\O_{\mathcal{E},\Prism}) \simeq \mathrm{Vect}(X_{qsyn},\Prism_\bullet[\frac{1}{I}]^\wedge_p).
$$
\end{lemma} 
\begin{proof}
The equivalence of $\mathrm{Vect}(X_{\Prism},\O_{\mathcal{E},\Prism})$ and $\mathrm{Vect}(X_{qsyn},\Prism_\bullet[\frac{1}{I}]^\wedge_p)$ is established in \cite[Proposition 2.14]{BS2021Fcrystals}. Here, one reduces to the case $X=\Spf(R)$ with $R$ quasi-syntomic, and the equivalence is given by the natural functor 
$$
\Prism(-)\colon R_{qrsp} \to  R_{\Prism}
$$
that maps $S$ to $\Prism_S$. 

Let $S$ be a quasiregular semiperfectoid covering $R$. In particular, we have $S$ is $p$-torsion free. This will imply that $\Prism_S$ is transversal, cf. \cite[Proposition 7.10]{BS19}. Let $S^\bullet$ be the \v{C}ech nerve of this map, then we have all three categories equivalent to
$$
\lim \mathrm{Vect}(\Prism_{S^\bullet})
$$
here we use the fact that $\Prism_{S^i}$ are transversal for all $i$ by Lemma~\ref{lem:transversalsite}.
\end{proof}

For any transversal prism $(A, I)$ and $n\in \N$ greater than $1$, let $I_n\coloneqq \prod_{i=1}^{n}\varphi^{i}(I)$. 

\begin{lemma}[{\cite[Lemma 3.3]{ALB-cyclotomic-trace}}]
$I_n$ defines a Cartier divisor on $\Spec(A)$.
\end{lemma}

\begin{remark}
Note that our definition of $I_n$ differs from a Frobenius twist with the one considered in \cite{ALB-cyclotomic-trace}.
\end{remark}

Define 
$$
\O_{\mathcal{E},\Prism}^{\dagger,n}(A,I)=A\langle\frac{p}{I_n}\rangle[\frac{1}{I_n}],
$$ 
where $A\langle\frac{p}{I_n}\rangle$ is the $(p,I_n)$-adic completion of $A[\frac{p}{I_n}]$. 

\begin{remark}\label{rem:adictop}
We have $\varphi(I)$ coincides with $I^p$ after modulo $p$, so for $n\in \N$, $(p,I_n)=(p,I^m)$ for some $m$, and
$$
(p,I)^m \subset (p,I^m) = (p,\prod_{i=1}^n\varphi^i(I)) \subset (p,I),
$$
that is $(p,I)$, $(p,\varphi(I))$ and $(p,I_n)$ defines the same topology on $A$. 

Over $A[\frac{p}{I_n}]$, we have $I_n$ divides $p$, so the $(p,I_n)$-adic topology is the same as the $I_n$-adic topology.
\end{remark}

\begin{lemma}\label{sheafnessofOdagger}
$\O_{\mathcal{E},\Prism}^{\dagger,n}$ defines a sheaf on $X_{\Prism^\circ}$.
\end{lemma}
\begin{proof}
Let $(A,I) \to (B,IB)$ be a covering in $X_{\Prism^\circ}$, we need to show 
$$
\O_{\mathcal{E},\Prism}^{\dagger,n}((A,I)) \to \O_{\mathcal{E},\Prism}^{\dagger,n}((B,IB)) \quad\substack{\rightarrow\\[0.05em] \rightarrow \\[-0.05em]}\quad \O_{\mathcal{E},\Prism}^{\dagger,n}((B\hat\otimes_A B,IB\hat\otimes_A B))
$$
is exact. By the definition of $\O_{\mathcal{E},\Prism}^{\dagger,n}$, it is enough to show the following lemma.
\end{proof}

\begin{lemma}\label{lem:A<>isasheaf}
We have $(A,I) \mapsto A\langle\frac{p}{I_n}\rangle$ defines a sheaf on $X_{\Prism^\circ}$.
\end{lemma}

And we will deduce Lemma~\ref{lem:A<>isasheaf} from the following two lemmas.
\begin{lemma}\label{lem:A<>toB<>iscompletefaithfullyflat}
Let $(A,I) \to (B,IB)$ be a covering of transversal prisms, then $A\langle \frac{p}{I_n}\rangle \to B\langle \frac{p}{I_n}\rangle
$ is $I_n$-complete faithfully flat.
\end{lemma}
\begin{proof}
First, by \cite[Lemma 3.6]{BS19}, $\varphi(I)$ is principle generated by a distinguished element $\xi$. And we have $I_n$ is generated by $\tilde{\xi}_n\coloneqq\prod_{i=0}^{n-1}\varphi^i(\xi)$. 

By discussions in \cite[page 11]{BS19} or \cite[\S4.1]{BMS2}, we have if an ideal $J$ of $A$ induces the same topology as $(p,I)$, then $A \to B$ is also $J$-complete flat. In particular, we have $A \to B$ is $(p,\tilde{\xi}_n)$-complete flat by Remark~\ref{rem:adictop}. Moreover it is also $(p,\tilde{\xi}_n)$-complete faithfully flat since $A/(p,\tilde{\xi}_n)$ and $A/(p,I)$ have the same spectrum. 

For $m_1,m_2\in \N$, consider $B_{m_1,m_2} \coloneqq B \otimes_A^{\mathbb{L}} A/(p^{m_1},\xi^{m_2})$. Then $B_{m_1,m_2} = (B \otimes_A^{\mathbb{L}} A/\xi^{m_2}) \otimes_{A/\xi^{m_2}}^{\mathbb{L}} A/(p^{m_1},\xi^{m_2})$. By \cite[Lemma 4.7]{BMS2} or the proof of (2) in \cite[Lemma 3.7]{BS19}, we have $B \otimes_A^{\mathbb{L}} A/\xi^{m_2}$ is discrete and equal to $B/\xi^{m_2}$. Now we can repeat this argument, using the fact that $B/\xi^{m_2}$ is $p$-complete flat over $A/\xi^{m_2}$, so $B_{m_1,m_2}$ is discrete and equal to $B/(p^{m_1},\xi^{m_2})$. 

Note that by \cite[Tag 00M8]{stacks-project}, one also have $A[X] \to B[X]$ is $(p,\xi)$-complete faithfully flat, and $(p,\tilde{\xi}_n)=(p,\xi^m)=(\tilde{\xi}_n,\tilde{\xi}_nX-p)$  in $A[X]$ for some $m\in \N$. So same argument shows 
\begin{IEEEeqnarray*}{rCl}
B[X] \otimes_{A[X]}^{\mathbb{L}} A[X]/(\tilde{\xi}_nX-p) & \simeq & B[X]/(\tilde{\xi}_nX-p)
\end{IEEEeqnarray*} 
and
\begin{IEEEeqnarray*}{rCl}
B[X] \otimes_{A[X]}^{\mathbb{L}} A[X]/(\tilde{\xi}_n^k,\tilde{\xi}_nX-p) & \simeq & B[X]/(\tilde{\xi}_n^k,\tilde{\xi}_nX-p) \IEEEyesnumber\label{eq:1}\\
& \simeq & B[\frac{p}{\tilde{\xi}_n}]/(p^k,\tilde{\xi}_n^k)=B[\frac{p}{\tilde{\xi}_n}]/(p^k,\xi^{km}).
\end{IEEEeqnarray*} 
When $k=1$, this implies $B\langle \frac{p}{\tilde{\xi}_n}\rangle/(\tilde{\xi}_n) \simeq B/(p,\tilde{\xi}_n)[X]$. So we have $A\langle \frac{p}{\tilde{\xi}_n}\rangle/(\tilde{\xi}_n) \to B\langle \frac{p}{\tilde{\xi}_n}\rangle/(\tilde{\xi}_n)$ is faithfully flat since $A \to B$ is $(p,\tilde{\xi}_n)$-complete faithfully flat.

It remains to show $\mathrm{Tor}^{A\langle \frac{p}{\tilde{\xi}_n}\rangle}_i(A\langle \frac{p}{\tilde{\xi}_n}\rangle/(\tilde{\xi}_n),B\langle \frac{p}{\tilde{\xi}_n}\rangle)=0 $ for $i>0$. It is enough to show this for $i=1$, we claim 
$$
\mathrm{Tor}^{A\langle \frac{p}{\tilde{\xi}_n}\rangle}_1(A\langle \frac{p}{\tilde{\xi}_n}\rangle/(\tilde{\xi}_n),B\langle \frac{p}{\tilde{\xi}_n}\rangle)= B\langle \frac{p}{\tilde{\xi}_n}\rangle [\tilde{\xi}_n]
$$
vanishes. This is because when $B$ is transversal, $B \subset B[\frac{p}{\tilde{\xi}_n}] \subset B[\frac{1}{\tilde{\xi}_n}]$ are $\tilde{\xi}_n$-torsion free, and the claim comes from the following fact: if $R$ is $d$-torsion free, let $\hat{x}=(x_n)_n \in \varprojlim_n R/d^n$ be an element in the completion such that $d\hat{x}=0$, then $\hat{x}=0$. Let $y_n\in R$ be any lifting of $x_n$, then $dy_n \in d^n R$ from the definition. Since $R$ is $d$-torsion free, we have $y_n \in d^{n-1} R$ which implies $x_{n-1}=0$. But $n$ is arbitrary, so $\hat{x}=0$.
\end{proof}

\begin{lemma}\label{lem:tensorbasechange}
Let $(A,I) \to (B,IB)$ be a covering of transversal prisms, and let $\tilde{\xi}_n$ be as in the proof of Lemma~\ref{lem:A<>toB<>iscompletefaithfullyflat}, then 
$$
(B\hat{\otimes}_A B) \langle \frac{p}{\tilde{\xi}_n} \rangle \simeq B\langle \frac{p}{\tilde{\xi}_n} \rangle \hat{\otimes}_{A \langle \frac{p}{\tilde{\xi}_n} \rangle} B\langle \frac{p}{\tilde{\xi}_n} \rangle
$$
where $-\hat{\otimes}-$ is the $(p,\tilde{\xi}_n)$-complete tensor product.
\end{lemma}
\begin{proof}
With all the notations used in the proof of Lemma~\ref{lem:A<>toB<>iscompletefaithfullyflat}, we have $A \to B$ is $(p,\xi)$-complete faithfully flat, so does $A \to B\hat{\otimes}_A B$. So if one replace $B$ with $B\hat{\otimes}_A B$ in Eq.~(\ref{eq:1}) then we will get
\begin{IEEEeqnarray*}{rCl}
(B\hat{\otimes}_A B)[X] \otimes_{A[X]} A[X]/(\tilde{\xi}_n^k,\tilde{\xi}_nX-p) & \simeq & (B\hat{\otimes}_A B)[\frac{p}{\tilde{\xi}_n}]/(p^k,\xi^{km}) \IEEEyesnumber\label{eq:2}\\
& \simeq & (B{\otimes}_A B)[X]/(p^k,\xi^{km},\tilde{\xi}_nX-p)
\end{IEEEeqnarray*} 
where the last isomorphism is by definition. On the other hand, we have 
$$
B\langle \frac{p}{\tilde{\xi}_n} \rangle \hat{\otimes}_{A \langle \frac{p}{\tilde{\xi}_n} \rangle} B\langle \frac{p}{\tilde{\xi}_n} \rangle \simeq \varprojlim_k \big( (B\langle \frac{p}{\tilde{\xi}_n} \rangle {\otimes}_{A \langle \frac{p}{\tilde{\xi}_n} \rangle} B\langle \frac{p}{\tilde{\xi}_n} \rangle)\otimes_{A \langle \frac{p}{\tilde{\xi}_n} \rangle} A[X]/(p^k,\xi^{km},\tilde{\xi}_nX-p) \big).
$$
Using commuting with base extension property of tensor product, we have
$$
\big(B\langle \frac{p}{\tilde{\xi}_n} \rangle {\otimes}_{A \langle \frac{p}{\tilde{\xi}_n} \rangle} B\langle \frac{p}{\tilde{\xi}_n} \rangle \big)\otimes_{A \langle \frac{p}{\tilde{\xi}_n} \rangle} \big(A[X]/(p^k,\xi^{km},\tilde{\xi}_nX-p)\big)
$$
is isomorphic to
$$
B[X]/(p^k,\xi^{km},\tilde{\xi}_nX-p) \otimes_{A[X]/(p^k,\xi^{km},\tilde{\xi}_nX-p)} B[X]/(p^k,\xi^{km},\tilde{\xi}_nX-p)
$$
here we also use Eq.~\ref{eq:1}. Moreover, this is also isomorphic to Eq.~\ref{eq:2}.
\end{proof}

\begin{proof}[Proof of Lemma~\ref{lem:A<>isasheaf}]
We have $A\langle \frac{p}{I_n}\rangle \to B\langle \frac{p}{I_n}\rangle
$ is $I_n$-complete faithfully flat by Lemma~\ref{lem:A<>toB<>iscompletefaithfullyflat}. So Lemma~\ref{lem:A<>isasheaf} follows from Lemma~\ref{lem:tensorbasechange} and $I_n$-complete faithfully flat descent.
\end{proof}

\begin{definition}
A crystal in finite free $\O_{\mathcal{E},\Prism}^{\dagger,n}$-modules is a sheaf $\mathcal{M}_{\Prism}$ of locally finite free $\O_{\mathcal{E},\Prism}^{\dagger,n}$-modules such that for all $(A,I) \to (B,IB)$ of prisms over $X$, we have the canonical map
$$
\mathcal{M}_{\Prism}((A,I))\widehat \otimes_{A} B \to \mathcal{M}_{\Prism}((B,IB))
$$
is an isomorphism. We will denote this category by $\mathrm{Vect}(X_{\Prism^\circ},\O_{\mathcal{E},\Prism}^{\dagger,n})$.
\end{definition}

\begin{remark}
Using \cite[Proposition 2.4.1]{Bhatt-Lurie-Absoluteprismaticcohomology}, one could extend a crystal in $\mathrm{Vect}(X_{\Prism^\circ},\O_{\mathcal{E},\Prism}^{\dagger,n})$ to a crystal in $\mathrm{Vect}(X_{\Prism},\O_{\mathcal{E},\Prism}^{\dagger,n})$. But we will actually only work with transversal (even oriented transversal) prisms over $X$ in this paper.
\end{remark}

\begin{corollary}\label{cor:descent}
Let $(A,I) \to (B,IB)$ be a covering of transversal prisms, then we have
\begin{IEEEeqnarray}{rCl}\label{eq:3}
   \mathrm{Vect}(\O_{\mathcal{E},\Prism}^{\dagger,n}(A,I)) \xrightarrow{\sim} \varprojlim\big(\mathrm{Vect}(\O_{\mathcal{E},\Prism}^{\dagger,n}(B,I)) \quad\substack{\rightarrow\\[0.05em] \rightarrow \\[-0.05em]}\quad \mathrm{Vect}(\O_{\mathcal{E},\Prism}^{\dagger,n}(B\widehat{\otimes}_A B)) \quad\substack{\rightarrow\\[-1em] \rightarrow \\[-1em] \rightarrow} \quad \cdots\big) 
\end{IEEEeqnarray}
here $\mathrm{Vect}(R)$ is the category of finite projective modules of $R$.
\end{corollary}
\begin{proof}
This also follows from Lemma~\ref{lem:A<>toB<>iscompletefaithfullyflat} and Lemma~\ref{lem:tensorbasechange}, then apply \cite[Theorem 5.8]{Mathew2020faithfully} or (3) in \cite[Theorem 2.2]{BS2021Fcrystals}.
\end{proof}

We have for $n>0$, there is a natural map $\iota_n\colon\O_{\mathcal{E},\Prism}^{\dagger,{n}} \to \O_{\mathcal{E},\Prism}^{\dagger,{n+1}}$, and we define
$$
\O_{\mathcal{E},\Prism}^{\dagger}\coloneqq\colim_{\iota_n,n\geq 1} \O_{\mathcal{E},\Prism}^{\dagger,{n}}
$$
as sheaves over $X_{\Prism^\circ}$. On the other hand, we have $\varphi_\Prism$ on $\O_\Prism$ defines a map $$
\varphi_n\colon\O_{\mathcal{E},\Prism}^{\dagger,{n}} \to \O_{\mathcal{E},\Prism}^{\dagger,{n+1}}
$$
which induces a map $\varphi_\Prism\colon\O_{\mathcal{E},\Prism}^{\dagger} \to \O_{\mathcal{E},\Prism}^{\dagger}$.

\begin{definition}\label{def:overconvergentFcrystal}
An overconvergent Laurent $F$-crystal over $X$ is a crystal $\mathcal{M}_\Prism$ in finite free $\O_{\mathcal{E},\Prism}^{\dagger}$-modules together with an isomorphism
$$
\varphi_{\mathcal{M}_\Prism}\colon\varphi_\Prism^\ast \mathcal{M}_\Prism \simeq \mathcal{M}_\Prism.
$$
And we let $\mathrm{Vect}(X_{\Prism^\circ},\O_{\mathcal{E},\Prism}^{\dagger})^{\varphi=1}$ to be category of overconvergent Laurent $F$-crystals over $X$.
\end{definition}

Let's recall the following general facts about sections on colimits of sheaves.
\begin{lemma}[{\cite[Tag 0738]{stacks-project}}]\label{lem:colimitcommutewithsections}
Let $\mathcal{C}$ be a site. Let $I \to \Shv(\mathcal{C})$, $i \mapsto \mathcal{F}_i$ be a filtered diagram of
sheaves of sets. Let $U$ be a quasi-compact object of $\mathcal{C}$ in the sense of \cite[Tag 090H]{stacks-project}. Consider the canonical map
$$
\Psi \colon \colim_{i} \mathcal{F}_i(U) \to (\colim_i \mathcal{F}_i) (U),
$$
where $(\colim_i \mathcal{F}_i)$ is taken inside $\Shv(\mathcal{C})$, then $\Psi$ is an isomorphism if all transition maps over sections over $U$ are injective.
\end{lemma}

\begin{remark}
We will see that all objects in $X_{\Prism^\circ}$ are quasi-compact by definition. Moreover, all transition maps $\iota_n$ on sections are injective by Corollary~\ref{cor:injectiveonsections}, so Lemma~\ref{lem:colimitcommutewithsections} implies taking colimits commutes with taking sections for our definition of $\O_{\mathcal{E},\Prism}^{\dagger,n}$.
\end{remark}

\subsection{Noetherian approximation via Breuil--Kisin prism}\label{sec:BKprism}
Objects in $X_{\Prism^\circ}$ can be very random, which makes it very hard to understand sections of $\O_{\mathcal{E},\Prism}^{\dagger,n}$ over them. The strategy we are going to explain in this subsection is that we first understand sections over prisms can be ``approximated" by Noetherian ones, or more precisely, the Breuil-Kisin prism, then descend properties from there to general transversal prisms. 

In this subsection, we only work on the case that $X=\Spf(\O_K)$ with $K$ a $p$-adic field, but all results in this subsection generalize to smooth affine formal schemes $\Spf(R)$ that are small, where we can also define the relative Breuil--Kisin modules as in \cite{DLMS-smoothcase} using the framing map. Recall the Breuil--Kisin prism $(\mathfrak{S},E)$ inside $(\O_K)_\Prism$ is defined as $\mathfrak{S}=W[\![u]\!]$ and $E$ is an Eisenstein polynomial over $W$ that is the minimal polynomial of a fixed uniformizer $\varpi \in \O_K$. Let's recall some basic properties of the Breuil--Kisin prisms. 

\begin{proposition}
$(\mathfrak{S},E)$ covers the final object $\ast$ in $\Shv(X_\Prism)$ and $\Shv(X_{\Prism^\circ})$.
\end{proposition}
\begin{proof}
This can be show using \cite[Lemma 4.1]{BS19}, or more concretely one can show the product of $(\mathfrak{S},E)$ with $(A,I)$ in $\Shv(X_\Prism)$ can be computed as in Example 2.6 of $loc.$ $cit.$ that we have for any prism $(A,I)\in X_\Prism$, we have the product of $(A,I)$ and $(\mathfrak{S},E)$ is 
$$
A[u]\{\frac{u-v}{I}\}^\wedge_{(p,I)}
$$
for some $v$ in $A$ that lifts the image of $\varpi$ under the structure map $\O_K \to A/I$.
\end{proof}

\begin{corollary}\label{cor:doublecoveringcoproduct}
Let $(A,I)\in X_{\Prism^\circ}$, and let $(B,EB)$ be the product of $(A,I)$ and $(\mathfrak{S},E)$ in $X_\Prism$, then the natural map $\mathfrak{S}\to B$ is classically fully faithful.
\end{corollary}
\begin{proof}
First, recall that we have $IB=EB$ by the irreducibility of distinguished elements. We have $(B,IB)$ covers $(A,I)$ from the last proposition, so by Lemma~\ref{lem:transversalsite}, we have $(B,IB)$ is also transversal, i.e., $B/EB$ is $p$-torsion free, which will imply that $\mathfrak{S}/E \to B/E$ is flat. Moreover, one can also check that $(B,EB)$ is also a covering of $(\mathfrak{S},E)$ by the $E$-torsionfreeness of $B$. Now we have $\mathfrak{S}\to B$ is also classically faithfully flat since $\mathfrak{S}$ is Noetherian and using \cite[Tag 0912]{stacks-project}.
\end{proof}

\begin{remark}
One should compare the above corollary with \cite[Proposition 2.4.9]{Bhatt-Lurie-Absoluteprismaticcohomology}.
\end{remark}

\begin{lemma}\label{lem:injectiveofBK}
Over the Breuil--Kisin prism 
$$
\iota_{n,\mathfrak{S}}\colon\O_{\mathcal{E},\Prism}^{\dagger,n}((\mathfrak{S},E)) \to \O_{\mathcal{E},\Prism}^{\dagger,n+1}((\mathfrak{S},E))
$$
is injective for all $n$.
\end{lemma}

Assume the above lemma, we will have

\begin{corollary}\label{cor:injectiveonsections}
For any $(A,I) \in X_{\Prism^\circ}$, $\iota_{n,A}$ is injective for all $n$. In particular by Lemma~\ref{lem:colimitcommutewithsections}, we will have 
$$
\O_{\mathcal{E},\Prism}^{\dagger}((A,I)) = \colim_n\O_{\mathcal{E},\Prism}^{\dagger,n}((A,I)).
$$
\end{corollary}
\begin{proof}
Let $(B,IB)$ be as in Corollary~\ref{cor:doublecoveringcoproduct}, we have $\mathfrak{S}\to B$ is faithfully flat, we first show that $\iota_{n,B}$ are injective for all $n$. We can rewrite elements in $B\langle \frac{p}{I_n} \rangle$ inside $B\langle \frac{p}{u^N} \rangle$ for some $N$, cf. Lemma~\ref{lem:keylemma}. So it is enough to show $B\langle \frac{p}{u^N} \rangle \to B\langle \frac{p}{u^{N+1}} \rangle$ is injective for all $N\in \N$. But this map is injective since it modulo $u$ is injective by Lemma~\ref{lem:injectiveofBK} and the flatness of $\mathfrak{S} \to B$. Moreover, $B\langle \frac{p}{u^N}\rangle$ is $u$-separated, it is also $u$-torsion free using a similar argument as in the last part of the proof of Lemma~\ref{lem:A<>toB<>iscompletefaithfullyflat}.

To descent this to sections over $(A,I)$, let's write $S_n=B\langle \frac{p}{I_n} \rangle$ and $R_n= A\langle \frac{p}{I_n} \rangle$, and for any ring $R$ over $A$, we write $\overline{R}$ as the reduction of $R$ by $I_n$, then by the proof of Lemma~\ref{lem:A<>toB<>iscompletefaithfullyflat}, we have $\overline{S}_n=\overline{R}_n\otimes_{\overline{A}}\overline{B}$. Let $K$ be the kernel of the map $\overline{R}_n\to \overline{R}_{n+1}$, using the fact that $\overline{A} \to \overline{B}$ is $p$-complete faithfully flat, we have $K \equiv 0 \mod p$. On the other hand, we have $\overline{A}$ is classically $p$-completed, so in particular, $p$-adic separated, so $K=0$. Repeat the last argument and using the fact that $R_n$ is classically $(p,I)$-adically complete, we have $R_n \to R_{n+1}$ is also injective.
\end{proof}

\begin{remark}
As a result, we have Corollary~\ref{cor:descent} still holds when replacing $\O_{\mathcal{E},\Prism}^{\dagger,n}$ with $\O_{\mathcal{E},\Prism}^{\dagger}$ since by Corollary~\ref{cor:injectiveonsections}, any finite projective module over $\O_{\mathcal{E},\Prism}^{\dagger}(A,I)$ is defined over $\O_{\mathcal{E},\Prism}^{\dagger,n}(A,I)$ for some $n$. Moreover, again by Corollary~\ref{cor:descent}, this ``convergent radius" is universally defined over all sections in $X_{\Prism^\circ}$.
\end{remark}

\begin{lemma}\label{lem:keylemma}
For a nonzero prism $(A,I)\in X_{\Prism^\circ}$, assume that an oriented prism $(A,I)$ is a covering of a prism which is Noetherian, then there are $n_0 < n_1$ depend on $n$, such that $n_0$ goes to $\infty$ as $n$ goes to $\infty$. When $n>1$, $n_0,n_1$ satisfy 
$$
A\langle \frac{p}{{I}^{n_0}} \rangle[\frac{1}{I}] \subset \O_{\mathcal{E},\Prism}^{\dagger,n}((A,I)) \subset A\langle \frac{p}{{I}^{n_1}} \rangle[\frac{1}{I}].
$$
\end{lemma}
\begin{proof}
We may assume $I=(d)$. Then by assumption, we have there is ($(p,I)$-complete) faithfully flat map $A_0 \to A$ with $A_0$ being Noetherian. Then we have $A[p/d]$ is $d$-adically separated since $A_0[p/I]$ is. So in particular, we can view $A\langle p/d \rangle$ as the closure of $A[p/d]$ for the $d$-adic topology, and elements inside $A\langle p/d \rangle$ can be represented by Cauchy series $\sum_{i\geq 0}a_i(p/d)^i$ with $a_i$ goes to $0$ $d$-adically. Here the last part is from the fact that if $\{x_n\}_{n\geq 0} \subset A[p/d]$ converges $d$-adically to some element that is not in $A[p/d]$, we can always assume that $x_n$ is unbounded for the $d$-adic topology inside $A[1/d]$.

We will prove a slightly general formula that will be used in our later content. Assume we can find some $x \in (p,d)$ such that 
$$
\varphi^{n}(d)=x^{p^n}+pu_n
$$
with $u_n$ being invertible for all $n\geq 1$. When $n=1$, the above equation implies that $(x,p)$ and $(d,p)$ define the same topology. For the statement of this lemma, one can check $x=d$ fulfill this requirement. 

To prove this result, we have
$$
\frac{p}{\prod_{i=1}^n (x^{p^i}+pu_i)}=\frac{p}{\prod_{i=1}^n x^{p^i}}\cdot\frac{1}{\prod_{i=1}^n(1+u_i\frac{p}{x^{p^i}})},
$$
And if we let $n_1= \sum_{i=1}^n p^i+n$, we will have 
$$
A\langle \frac{p}{\prod_{i=1}^n (x^{p^i}+pu_i)} \rangle \subset A\langle \frac{p}{{x}^{n_1}}\rangle.
$$

On the other hand, if we let $n_0 = \sum_{i=1}^n p^i$, we can show
$$
\frac{p}{x^{n_0}}=\frac{p}{\prod_{i=1}^n(x^{p^i}+pu_i-pu_i)}=\frac{p}{\prod_{i=1}^n(x^{p^i}+pu_i)}\frac{1}{\prod_{i=1}^n(1-u_i\frac{p}{x^{p^i}+pu_i})}.
$$
For each $i$ between $1$ and $n$, we have
$$
\frac{1}{(1-u_i\frac{p}{x^{p^i}+pu_i})}=\sum_{m=0}^\infty u_i^m\big(\frac{p}{x^{p^i}+pu_i}\big)^m = \sum_{m=0}^\infty u_i^m\big(\frac{p\prod_{0\leq k \leq n, k\neq i}  (x^{p^k}+pu_k)}{\prod_{k=0}^n( x^{p^k}+pu_k)}\big)^m
$$
which is inside $A\langle \frac{p}{\prod_{i=1}^n (x^{p^i}+pu_i)}\rangle$. So this implies
$$
A\langle \frac{p}{{x}^{n_0}}\rangle[\frac{1}{x}] \subset A\langle \frac{p}{\prod_{i=1}^n (x^{p^i}+pu_i)}\rangle[\frac{1}{\prod_{i=1}^n (x^{p^i}+pu_i)}].
$$
\end{proof}

\begin{proof}[{Proof of Lemma~\ref{lem:injectiveofBK}}]
In the proof of Lemma~\ref{lem:keylemma}, if we let $(A,I)=(\mathfrak{S},E)$, and we have 
$$
\varphi^n(E)=x^{p^n}+pu_n
$$
with $x=u^e$. So we have for all $n\geq 1$, we can regards $\O_{\mathcal{E},\Prism}^{\dagger,n}((\mathfrak{S},E))$ and $\O_{\mathcal{E},\Prism}^{\dagger,n+1}((\mathfrak{S},E))$ as subset of $\mathfrak{S}\langle \frac{p}{u^{N}} \rangle[\frac{1}{u}]$ for some $N\gg 0$. Then it is not hard to see $\O_{\mathcal{E},\Prism}^{\dagger,n}((\mathfrak{S},E))\subset\O_{\mathcal{E},\Prism}^{\dagger,n+1}((\mathfrak{S},E))$.
\end{proof} 

With the notions in the proof of Lemma~\ref{lem:keylemma}, if we write 
$$
\O_{\mathcal{E}}\coloneqq\O_{\mathcal{E},\Prism}((\mathfrak{S},E))= \left \{ \, \sum_{i\in\Z} \, a_i u^i\,\mid\, a_i \in W, \, \lim_{i\to -\infty} a_i=0\,\right \},
$$
then we will have $\O_\E^{\dagger}\coloneqq\O_{\mathcal{E},\Prism}^{\dagger,n}((\mathfrak{S},E)) \subset \O_{\mathcal{E}}$. We have the following generalizations of the above inclusion. For $k \in \N$, let $(\mathfrak{S}^{k},E)$ be the $k+1$-th self-product of $(\mathfrak{S},E)$ in $X_\Prism$, then we can have the natural projections $p_j: \mathfrak{S} \to \mathfrak{S}^{k}$ are all classically faithfully flat, cf., \cite[Proposition 2.2.7]{Du-Liu-phiGhatmodules} and \cite[Lemma 3.5]{DLMS-smoothcase}. Let $x \in \O_{\E, \Prism}^{\dagger} (\mathfrak S ^k,E)$, then by Lemma~\ref{lem:keylemma}, there is $N \gg 0$ and we can write $E^Nx=x^+$ with $x^+= \sum\limits_{i = 0}^\infty a_i (\frac{p}{E^n})^i$ for some $n\in \N$ with $\{a_i\}$ converges $(p, E)$-adically to $0$ in $\mathfrak S ^k$. Since $E ^{ns} (\frac{p}{ E ^n}) ^i = p ^{s} (\frac{p}{E^n})^{i -s}$, we may rewrite the terms of $x^+$ so that $a_i \to 0$ $p$-adically in $\mathfrak S^1$. Rewriting $x$ in this way induces a map $c_k\colon\O_{\E, \Prism}^{\dagger} (\mathfrak S ^k,E) \to \O_{\E, \Prism} (\mathfrak S ^k,E)$.

One could check, as we mentioned at the beginning of this subsection, that if $R$ is $p$-complete \textit{small} $\O_K$-algebra in the sense of \cite[Definition 2.1]{DLMS-smoothcase}, we could also define the (relative) Breuil--Kisin prism $(\mathfrak{S},E)$ as in Example 3.4 of $loc.$ $cit.$, and we have the same results carried out for the transversal prismatic site over $X=\Spf(R)$ and $(\mathfrak{S}, E)$. In particular, the maps $c_k$ are well-defined.

Now an object in $\mathrm{Vect}(X_{\Prism},\O_{\mathcal{E},\Prism})^{\varphi=1}$ (resp. $\mathrm{Vect}(X_{\Prism^\circ},\O_{\mathcal{E},\Prism}^{\dagger})^{\varphi=1}$) is equivalent to a pair $(\mathcal{M},f)$ with $\mathcal{M}$ an \'etale $\varphi$-module over $\O_\E$ (resp. $\O_{\mathcal{E}}^{\dagger}$) together with an $\varphi$-equivariant descent isomorphism over $\O_{\mathcal{E},\Prism}((\mathfrak{S}^1,E))$ (resp. $\O_{\mathcal{E},\Prism}^{\dagger,n}((\mathfrak{S}^1,E))$). Then base change along $c_k$ defines a functor:
\begin{equation}\label{eq:main}
D_X: \mathrm{Vect}(X_{\Prism^\circ},\O_{\mathcal{E},\Prism}^{\dagger})^{\varphi=1} \to \mathrm{Vect}(X_{\Prism},\O_{\mathcal{E},\Prism})^{\varphi=1}.
\end{equation}
And we can restrict this functor over the perfect site, and we get
\begin{equation}\label{eq:mainperf}
D_X^{\perf}: \mathrm{Vect}(X_{\Prism^\circ}^{\perf},\O_{\mathcal{E},\Prism}^{\dagger})^{\varphi=1} \to \mathrm{Vect}(X_\Prism^{\perf},\O_{\mathcal{E},\Prism})^{\varphi=1}.
\end{equation}

We made the following conjecture:
\begin{conjecture}\label{conj:main}
If $R$ is $p$-adically complete \textit{small} $\O_K$-algebra, and let $X=\Spf(R)$, then $D_X$ (resp. $D_X^{\perf}$) defines an equivalence of categories.
\end{conjecture}

\begin{remark}
One could define sheaves $\O_{\E,qrsp}\coloneqq \Prism_{\bullet}[1/I]^\wedge_p$ and $\O_{\E,qrsp}^\dagger\coloneqq \colim_n\Prism_\bullet\langle \frac{p}{I_n}\rangle[\frac{1}{I_n}]$  over $X_{qrsp}$ and $X_{qrsp}^\circ$ respectively similar as in \cite[\S6]{BS2021Fcrystals}. Here $X_{qrsp}$ is defined in Definition 2.9 of $loc.$ $cit.$ and $X_{qrsp^\circ}$ is the full subcategory of $X_{qrsp}$ consisting of quasi-regular semiperfectoid rings over $X$ that are $p$-torsion free. One could also define $\mathrm{Vect}(X_{qrsp^\circ},\O_{\mathcal{E},qrsp}^{\dagger})^{\varphi=1}$ similar to Definition~\ref{def:overconvergentFcrystal}, and the comparison result in \cite[
Proposition 2.1]{BS2021Fcrystals} gives an equivalence 
$$
\mathrm{Vect}(X_{qrsp^\circ},\O_{\mathcal{E},qrsp}^{\dagger})^{\varphi=1} \simeq \mathrm{Vect}(X_{\Prism^\circ},\O_{\mathcal{E},\Prism}^{\dagger})^{\varphi=1}.
$$
In particular Conjecture~\ref{conj:main} has a version over the quasi-syntomic sites. We do not pursue this direction here.
\end{remark}

\begin{remark}\label{rem:comparewithclassical}
From Lemma~\ref{lem:keylemma}, when $X=\Spf(\O_K)$, we have $\O_\E^{\dagger}$ defined as above is equal to the overconvergent period ring used in the theory of $(\varphi,\tau)$-modules. Beyond the Kummer tower case, one could compare $\O_{\mathcal{E},\Prism}^{\dagger,n}((A,d))$ with the period rings used in the work of Kedlaya-Liu in \cite[\S5]{KedlayaLiu-relativeII} or \cite{Kedlaya2013Rational} defined using $\theta$-maps.
\end{remark}

\section{Overconvergence of Laurent $F$-crystals}
In this section, we will assume $X$ to be $\Spf(\O_K)$. But some results, especially those in \S\ref{sec:perfectdecomp} carries over more general bases.
\subsection{Perfect overconvergent theory}\label{sec:perfectdecomp}
We consider the transversal perfect prismatic site $X_{\Prism^\circ}^{\perf}$ over $X$, i.e., objects in $X_{\Prism^\circ}^{\perf}$ are perfect prisms $(A,I)$ over $\O_K$ such that $A/I$ is $\O_K$-flat. Let $(\Ainf,\xi)$ be the Fontaine prism over $\O_K$.

\begin{lemma}
The sheaf represented by $(\Ainf,\xi)$ covers the final object in $X_{\Prism^\circ}^{\perf}$.
\end{lemma}
\begin{proof}
It is a well-known fact that the sheaf represented by $(\Ainf,\xi)$ covers the final object in $X_{\Prism}^{\perf}$, which can be directly deduced from \cite[Proposition 7.10 and Proposition 7.11]{BS19}. Then the result follows from Proposition~\ref{prop:transversal}.
\end{proof}

It is known that the category of perfect prisms over $X$ is equivalent to the category of integral perfectoid rings over $\O_K$, so it admits finite non-empty coproducts. We will denote $((\Ainf^\bullet)_{\perf},\xi)$ be the \v{C}ech nerve of $(\Ainf,\xi)$ in $X_\Prism$. One can identify $\Ainf^i$ with $\Prism_{R^i}$ where $R^i$ is the $i$-th complete self-tensor product of $\O_C$ over $\O_K$, and $(\Ainf^i)_{\perf}$ with $(\Prism_{R^i})_{\perf}$, then one can apply \cite[Proposition 7.10 and Proposition 7.11]{BS19} to see $((\Ainf^i)_{\perf},\xi)$ are transversal for $i>0$. For any topological ring, let $C^0(G_K^i, R)$ be the ring of continuous functions from $G_K^i$ to $R$. Let $\tilde{\xi}_n=\prod_{i=1}^n\varphi^i(\xi)$.

\begin{lemma}\label{lem:isomAinfi}
\begin{enumerate}
    \item $(\Ainf^i)_{\perf}$ is almost isomorphic to $C^0(G_K^i,\Ainf)$ for $i>0$;
    \item $(\Ainf^i)_{\perf}[1/\xi]^\wedge_p$ is isomorphic to $C^0(G_K^i,W(C^\flat))$;
    \item $(\Ainf^i)_{\perf}\langle \frac{p}{\tilde{\xi}_n}\rangle[\frac{1}{\tilde{\xi}_n}]$ is isomorphic to $C^0(G_K^i,\Ainf\langle \frac{p}{\tilde{\xi}_n}\rangle[\frac{1}{\tilde{\xi}_n}])$.
\end{enumerate}
\end{lemma}
\begin{proof}
We will have $C^0(G_K^i,\Ainf\langle \frac{p}{\tilde{\xi}_n}\rangle[\frac{1}{\tilde{\xi}_n}])\simeq C^0(G_K^i,\Ainf)\langle \frac{p}{\tilde{\xi}_n}\rangle[\frac{1}{\tilde{\xi}_n}]$ by the compactness of $G_K$, and one can check $(1)$ implies $(2)$ and $(3)$. We have $(1)$ follows from the theory of diamond and \cite[Example 3.22]{Scholze-Diamonds}. For the arguments using diamonds, one can refer to \cite[Lemma 5.3]{wu2021galois} or \cite[Lemma 2.13]{Min-Wang_relativephiGammamodules}.
\end{proof}

We will also denote $\Ainf$ by $\tilde{\mathbf{A}}^+$, and we let $\tbfA=W(\overline{K}^\flat)$. For $r>0$ let 
$$
\tbfA^{\dagger,r} = \{\sum_{n\geq 0}[a_n]p^n | a_n \in C^\flat \, ,\, a_n(p^{n/r}) \to 0  \},
$$
and $\tbfA^{\dagger}=\cup_{r>0} \tbfA^{\dagger,r}$. We have the following result.

\begin{theorem}\label{thm:perfectthm}
Base extension of \'etale $(\varphi,\Gamma)$-modules from $\tbfA^{\dagger}$ to $\tbfA$ is an equivalence of categories, and both categories are equivalent to the category of continuous representations of $G_K$ on finite free $\Z_p$-modules.
\end{theorem}
\begin{proof}
In the proof of Lemma~\ref{lem:keylemma}, if we let $d=[\underline{p}^\flat]-p$ and $x=[\underline{p}^\flat]$, then we will have for each $n\in \N$, there is $r,s>0$ such that 
$$
\tbfA^{\dagger,r} \subset \O_{\E,\Prism}^{\dagger,n}((\Ainf,\xi)) \subset \tbfA^{\dagger,s}.
$$
And the rest of this theorem becomes a well-known fact. Actually, we have the corresponding categories of \'etale $\varphi$-modules are both equivalent to the category of finite free $\Z_p$-modules. For statement about \'etale $\varphi$-modules, one can refer to \cite[Theorem 8.5.3]{KedlayaLiu-relativeI}. On the other hand, the statement about \'etale $(\varphi,\Gamma)$-modules is due to Lemma~\ref{lem:isomAinfi} and \cite[Theorem 2.4.5]{Kedlaya-newmethod}.
\end{proof}

\begin{corollary}
We have \eqref{eq:mainperf} is an equivalence of categories. 
\end{corollary}

\subsection{The deperfection}
In this section, we will show 
\begin{theorem}\label{thm:mainwithproof}
If $K$ is a finite extension of $\Q_p$, and let $X=\Spf(\O_K)$, we have \eqref{eq:main} is an equivalence of categories.
\end{theorem}

\begin{proof}[{Proof of fully faithfulness in Theorem~\ref{thm:mainwithproof}}]
It reduces to prove the functors 
\begin{equation}\label{eq:fullyfaithful-i}
\mathrm{Vect}(\O_{\mathcal{E},\Prism}^{\dagger}(\mathfrak{S}^i,E))^{\varphi=1} \to \mathrm{Vect}(\O_{\mathcal{E},\Prism}(\mathfrak{S}^i,E))^{\varphi=1} 
\end{equation}
induced by base change along $c_k$ defined in \S\ref{sec:BKprism} is fully faithful when $k=0$ and faithful when $k=1,2$, where $\mathrm{Vect}(R)^{\varphi=1}$ denotes the category of \'etale $\varphi$-modules over $R$. The fully faithfulness of \eqref{eq:fullyfaithful-i} when $k=0$ will follow from the fully faithfulness of the classical theory of \'etale $\varphi$-modules, for instance, \cite[Proposition 5.4.8]{KedlayaLiu-relativeII}. When $k=1,2$, it is enough to prove that $c_k$ are injective.

Let $x$ and $x^+$ be as in the discussions when defining $c_k$. To show the injectivity of $c_k$ we may assume $x=x^+$ which is equal to $\sum\limits_{i = 0}^\infty a_i (\frac{p}{E^n})^i$ for some $n\in \N$ with $\{a_i\}$ converges $p$-adically to $0$ in $\mathfrak S ^k$. Suppose that $c_k(x)=0$ in $\O_{\E, \Prism  }(\mathfrak S^k,E).$ For $m \in N$, set 
$$x_m=  \sum\limits_{i = 0}^m a_i (\frac{p}{E^n})^i.
$$ 
It is enough to show for any $l\in\N$, $x_m $ are in $p ^ l \mathfrak S^k [\frac{p}{E^n}]$ for $m$ sufficient large. Since $a_i$ converges to $0$ $p$-adically in $\mathfrak S^k$ and $x=0$ in $\O _{\E, \Prism} (\mathfrak S ^k,E) $, we can choose $m$ sufficiently large so that $x_m \in p ^{m +l}\mathfrak{S}^k[\frac{1}{E}] \subset p ^{m +l} \O _{\E, \Prism} (\mathfrak S ^k,E)$. So $y_m := E^{m n} x_m \in \mathfrak S^k \cap  p ^{m +l} \O _{\E, \Prism} (\mathfrak S ^k,E) $. Note that $\mathfrak S ^k / p \mathfrak S ^k $ injects to $\O_{\E, \Prism  }(\mathfrak S^k,E)/ p \O_{\E, \Prism  }(\mathfrak S^k,E) = \mathfrak S ^k / p \mathfrak S ^k [\frac 1 u]$ and $\mathfrak{S}^k$ is $p$-adically separated, both because $\mathfrak S ^k $ is faithfully flat over $\mathfrak S$ along any of the projection maps. Therefore, $y _m$ is in $\mathfrak S^k \cap p ^{m +l} \O_{\E , \Prism} (\mathfrak S^k,E)= p ^{m +l} \mathfrak S ^k$. Say $y_m = p ^ {m +l} z$ with $z \in \mathfrak S^k$. Then $x_m = p ^l(\frac{p}{E^n})^m \in p ^l \mathfrak S^k [\frac{p}{E^n}] $ as required. 
\end{proof}

We will use the same method introduced in \cite{Gao-Liu-Loosecrystallinelifts} to show the essential surjectivity, except we will replace the theory of $(\varphi,\hat{G})$-modules developed in \cite{liu-notelattice} with the theory of prismatic $(\varphi,\hat{G})$-modules developed in \cite{Du-Liu-phiGhatmodules}. Let $T$ be a $\Z_p$-representation of $G_K$ and let $(\mathcal{M},\varphi_{\mathcal{M}},\hat{G})$ be the $(\varphi,\tau)$-module associated with $T$ in the sense of \cite[Theorem 4.2.3]{Du-Liu-phiGhatmodules}. Let $T_n\coloneqq T/p^n T$, then we have $T_n$ correspondence to $\mathcal{M}_n=\mathcal{M}/p^n\mathcal{M}$ which is from the correspondences studied in \cite{Caruso-phitau}. 

We will summarize the result in \cite[\S6]{Gao-Liu-Loosecrystallinelifts} in the form we need. In the following, $M_d (R)$ denote the set if $d\times d$-matrices with entries in $R$. 

\begin{theorem}[{\cite{GL20}}]
There is a basis $\{e_j^{(n)}\}_{j=1}^d$ of $\mathcal{\mathcal{M}}_n$, such that $\{\lim_{n}e_j\}_{j=1}^d$ forms a basis of $M=\varprojlim \mathcal{M}_n$. For each $n$, there is a $\mathfrak{S}$-sublattice $\mathfrak{M}_{(n)}$ of $\mathcal{M}_n$, i.e, $\mathcal{M}_n=\mathfrak{M}_{(n)}[1/u]$. Moreover, $\mathfrak{M}_{(n)}$ is generated over $\{\mathfrak{e}_{(n),j}^{(i)} \mid j\leq n\}\subset \mathcal{M}_n$ as $\mathfrak{S}$-module, satisfying that
\begin{enumerate}
    \item $(e_j^{(n)})=(\mathfrak{e}^{(n)}_{(n),j}) Y_n^{-1}$ with $Y_n,Y_n^{-1} \in M_d(\mathfrak{S})$;
    \item $\mathfrak{e}^{(i)}_{(n),j}=\mathfrak{e}^{(n)}_{(n),j} Y_{i,n}$ for $Y_{i,n} \in M_d(\mathfrak{S}[p/u^{2h}])$ for some $h>0$;
    \item for each $n>0$, $\mathfrak{M}_{(n)}$ is a quotient of two finite free Kisin module $\mathfrak{M}$ (may different for each $n$) attached to a crystalline representation of $G_K$ in the sense of \cite[Definition 4.1.1]{Gao-Liu-Loosecrystallinelifts}.
\end{enumerate}
\end{theorem}

\begin{corollary}\label{cor:tauaction}
There is a constant $n=n(T,K)\in \N$, and a basis $\{e_j\}$ of $\mathcal{M}$ such that $\varphi_M$ (resp. $\tau$) acts on $\{e_j\}$ by a matrix inside $M_d(\O_{\mathcal{E},\Prism}^{\dagger,n}((\mathfrak{S},E)))$ (resp $M_d(\O_{\mathcal{E},\Prism}^{\dagger,n}((\mathfrak{S}^1,E)))$).
\end{corollary}
\begin{proof}
This is from \cite[\S6]{Gao-Liu-Loosecrystallinelifts} and replace the theory of $(\varphi,\hat{G})$-modules with the theory of prismatic $(\varphi,\hat{G})$-modules attached with crystalline representations. By \cite[Corollary 3.3.4]{Du-Liu-phiGhatmodules}, we have $\tau$ acts on $\mathfrak{e}^{(n)}_{(n),j}$ by a matrix in $M_d(\mathfrak{S}^1)$. Here the main difference between classical $(\varphi,\hat{G})$-modules with prismatic $(\varphi,\hat{G})$-modules is that they are differed by a Frobenius twist for the definition of $\hat{G}$-action, and this enables us to have a better estimation of the $\tau$-action.
\end{proof}

\begin{proof}[{Proof of Theorem~\ref{thm:mainwithproof}}]
It remains to prove the essential surjectivity of $D_X$. For any Laurent $F$-crystals, we attach it a prismatic \'etale $(\varphi, \tau)$-module as in \cite[\S4.2]{Du-Liu-phiGhatmodules}. Let's briefly recall our setup, we fix a topological generator $\tau$ of $\Gal(L/K_{1^\infty})$, we can define a map $\mathfrak{S}^1 \to \Ainf$ induced by the two maps $\mathfrak{S} \hookrightarrow \Ainf$ and $\mathfrak{S} \hookrightarrow \Ainf \xrightarrow[]{\tau} \Ainf$. This map is injective by Corollary 2.4.5 of $loc.$ $cit.$, and induces natural maps $\mathfrak{S}^1[1/E]^\wedge_p \to W(C^\flat)$ and $\O_{\mathcal{E},\Prism}^{\dagger}((\mathfrak{S}^1,E)) \to \O_{\mathcal{E},\Prism}^{\dagger}((\Ainf,\xi))$. We have a Laurent $F$-crystal is equivalent to an \'etale $\varphi$-module $\mathcal{M}$ over $\O_{\E}$ together with a descent isomorphism over $\mathfrak{S}^1[1/E]^\wedge_p$. Then in \S4.2 of $loc.$ $cit.$ we show that if we base change the descent isomorphism along $\mathfrak{S}^1[1/E]^\wedge_p \to W(C^\flat)$, then it is exactly the $\tau$-action in the definition of $(\varphi, \tau)$-modules. By Corollary~\ref{cor:tauaction}, we have the matrix defines the $\tau$-action actually has coefficients inside $\O_{\mathcal{E},\Prism}^{\dagger,n}((\mathfrak{S}^1,E))$ for some $n$. The linearization of this $\tau$-action defines a descent isomorphism of an overconvergent Laurent $F$-crystal.
\end{proof}

\section{Overconvergence of $(\varphi,\Gamma)$-modules associated with $\Phi$-iterate tower.}\label{sec:phiiterate}
Let $K$ be as in \S\ref{sec:BKprism}, and we fix a polynomial $\Phi$ in $W[X]$ such that $\Phi(X)\equiv X^p \mod p$ and $\Phi(0)=0$. We fix a uniformizer $\varpi \in \O_K$ and also fix a system $\vec{\nu}=\{\nu_i\}_{i\geq 0}$ in $\overline{K}$ with $\nu_0=\varpi$ and $\Phi(\nu_{n+1})=\nu_n$. Let $K_n=K(\nu_n)$, then define $K_\infty=K_{\infty,\vec{\nu}}=\cup_n K_n$ and $G_\infty=\Gal(\overline{K}/K_\infty)$. We define $\mathfrak{S}_{\Phi}=W[\![u]\!]$ and define the $\delta$-structure of $\mathfrak{S}_{\Phi}$ determined by $\varphi(u)=\Phi(u)$. Let $E(u)$ be the minimal polynomial of $\varpi$ over $K_0$, we have $(\mathfrak{S}_{\Phi},E)$ is a transversal prism over $\O_K$. We define $\iota\colon\mathfrak{S}_{\Phi} \to \Ainf$ by $u \mapsto \{[\underline{\nu}_n^\flat]\}_{\Phi}$, where $\{\cdot\}_{\Phi}$ is the unique set-theoretic section $\O_C^\flat \to  \Ainf$ to the reduction modulo $p$ and satisfying $\varphi(\{x\}_{\Phi}) = \Phi(\{x\}_{\Phi} )$ for all $x\in \O_C^\flat$, cf., \cite[Lemme 9.3]{Colmez-fini} and also \cite[Lemma 2.1.1]{Cais-Liu-F-crystals}. We will let $\O_{\E,\Phi}=\mathfrak{S}_{\Phi}[1/E]^\wedge_p$. Then $\iota$ defines a map of prisms $(\mathfrak{S}_{\Phi},E) \to (\Ainf,\xi)$ and also $\O_{\E,\Phi} \to W(C^\flat)$. One has $\mathfrak{S}_{\Phi} \subset \Ainf^{G_\infty}$.

\begin{definition}
For $\ast\in\{\emptyset, \dagger\}$, we define the category $\mathrm{Mod}^{\text{\'et},\ast}_{\Phi}(\varphi,\Gamma)$ of \'etale $(\varphi,\Gamma)$-modules associated with the $\Phi$-iterate tower $\vec{\nu}$ consists of triples $(\mathcal{M}^\ast,\varphi_{\mathcal{M}^\ast},\widehat{\mathcal{M}}^\ast)$ where $(\mathcal{M}^\ast,\varphi_{\mathcal{M}^\ast})$ is an \'etale $\varphi$-module over $\O_{\E,\Prism}^\ast((\mathfrak{S}_{\Phi},E))$, and $\widehat{\mathcal{M}}^\ast\coloneqq\mathcal{M}^\ast\otimes_{\iota}\O_{\E,\Prism}^\ast((\Ainf,\xi))$ is equipped with a semilinear $G_K$-action such that regarding $\mathcal{M}^\ast$ as an $\O_{\E,\Prism}^\ast((\mathfrak{S}_{\Phi},E))$-submodule in $\widehat{\mathcal{M}}^\ast$, we have $\mathcal{M}^\ast \subset (\widehat{\mathcal{M}}^\ast)^{G_\infty}$.
\end{definition}

Given $(\mathcal{M},\varphi_{\mathcal{M}},\widehat{\mathcal{M}}) \in \mathrm{Mod}^{\text{\'et}}_{\Phi}(\varphi,\Gamma)$ associated with the $\Phi$-iterate tower $\vec{\nu}$ as above, we define $T= \widehat{\mathcal{M}}^{\varphi=1}$. Given $T\in \RepZp(G_K)$, a $\Z_p$-Galois representation of $G_K$, and for $\ast\in\{\emptyset, \dagger\}$, we define $(\mathcal{M}^\ast(T|_{G_\infty}),\varphi_{\mathcal{M}^\ast(T|_{G_\infty})}, \widehat{\mathcal{M}}^\ast(T))$ with $\widehat{\mathcal{M}}^\ast=T\otimes_{\Z_p}\O_{\E,\Prism}^\ast((\Ainf,\xi))$ and $(\mathcal{M}^\ast(T|_{G_\infty}),\varphi_{\mathcal{M}^\ast(T|_{G_\infty})})$ the \'etale $\varphi$-module over  associated with $T|_{G_\infty}$ as in \cite[\S2.5]{Gao-Liu-Loosecrystallinelifts}. One can show $T$ induces an equivalence of $\mathrm{Mod}^{\text{\'et}}_{\Phi}(\varphi,\Gamma)$ and $\RepZp(G_K)$ using \cite[Theorem 3.2.2]{Cais-Liu-F-crystals}. 

\begin{theorem}\label{thm:phiiterat}
Bese change induces an equivalence of categories between $\mathrm{Mod}^{\text{\'et},\dagger}_{\Phi}(\varphi,\Gamma)$ and $\mathrm{Mod}^{\text{\'et}}_{\Phi}(\varphi,\Gamma)$.
\end{theorem}
\begin{proof}
We will show all objects in $\mathrm{Mod}^{\text{\'et}}_{\Phi}(\varphi,\Gamma)$ are overconvergent in the sense that if $(\mathcal{M},\varphi_{\mathcal{M}},\widehat{\mathcal{M}})$ is associated with $T\in \RepZp(G_K)$, then $(\mathcal{M},\varphi_{\mathcal{M}},\widehat{\mathcal{M}})$ comes from the base change from $(\mathcal{M}^\dagger(T|_{G_\infty}),\varphi_{\mathcal{M}^\dagger(T|_{G_\infty})}, \widehat{\mathcal{M}}^\dagger(T))$. This follows from Theorem~\ref{thm:mainwithproof} by evaluation an overconvergent Laurent $F$-crystal to the diagram $(\mathfrak{S}_{\Phi},E) \to (\Ainf,\xi)$ and using $\mathcal{M}^\ast \subset (\widehat{\mathcal{M}}^\ast)^{G_\infty}$.
\end{proof}

\begin{corollary}
There is a uniform bound for the overconvergence radii of overconvergent \'etale $(\varphi,\Gamma)$-modules associated with the same Galois representation for different $\Phi$-iterate towers.
\end{corollary}
\begin{proof}
This comes from the proof of Theorem~\ref{thm:phiiterat} and the fact that for an overconvergent Laurent $F$-crystal, since it is finitely generated, it is defined over $\O_{\mathcal{E},\Prism}^{\dagger,n}$ for some $n$.
\end{proof}

\appendix

\section{Arithmetic Kummer tower is not decompleting}\label{app}
We will show arithmetic Kummer tower over $\Q_p$ is not decompleting in the sense of \cite[Definition 5.6.1]{KedlayaLiu-relativeII} by showing if it is decompleting, then all $p$-adic representations of $G_{\Q_p}$ are potentially unramified. In other words \cite[Conjecture 6.2.6]{KedlayaLiu-relativeII} does not hold.

Let $K=\Q_p$ and $K_n=K(\varpi_n)$, where $\varpi_0=p$ and $\{\varpi_n\}$ satisfies $\varpi_n=\varpi_{n+1}^p$. We define $K_n'$ be the normalizations of $K_n$ in $\overline{K}$. Let $L=\cup K_{n}$ and $L'=\cup K_{n}'$. Then $\{(K_n, \O_{K_{n}})\}$ and $\{(K_n', \O_{K_{n}'})\}$ form perfectoid finite \'etale towers over $K$ in the sense of \cite[\S5.1]{KedlayaLiu-relativeII}, so we can construct the period rings of type $\mathbf{A}$ for this two towers (where we let $E=\Z_p$ as in your notions), $\widetilde{\mathbf{A}}^*$ for the $\{K_n\}$ tower and $\widetilde{\mathbf{A}}'^*$ for the $\{K'_n\}$ tower, and also define their overconvergent version $\mathbf{A}^*$ and $\mathbf{A}'^*$, with $*\in\{r>0, \dagger\}$ as been defined in \S 5.2.1. of $loc.$ $cit$. 

Now let's assume that Kummer towers are decompleting, i.e., \cite[Conjecture 6.2.6]{KedlayaLiu-relativeII} holds, then we should have Theorem 5.7.3 of $loc.$ $cit.$ holds so that the category of ($\varphi, \Gamma$)-modules over $\mathbf{A}^\dagger$ and the category of ($\varphi, \Gamma$)-modules over $\widetilde{\mathbf{A}}^\dagger$ are equivalent via the base change functor. If we use the interpretations in Example 5.5.7 of $loc.$ $cit$., ($\varphi, \Gamma$)-modules over $\mathbf{A}^\dagger$ is the same as a finite projective $\varphi$-module $M$ over $\mathbf{A}^\dagger$ together with a semilinear $\Gal(L'/K)$ action on $M\otimes_{\mathbf{A}^\dagger}\mathbf{A}'^\dagger$ commutes with $\varphi$ and whose restriction to $\Gal(L'/L)$ fixes $M$. Let's assume $M$ is actually defined over $\mathbf{A}^r$, and we choose $n\gg 0$ such that $\theta\circ \varphi^{-n}$ is defined, and let $T$ be the $\Z_p$-Galois representation corresponds to $M$, then we will have $T\otimes_{\Z_p}\C_p=M\otimes_{\theta\circ \varphi^{-n}} \C_p$. Moreover, we have by the definition of $\mathbf{A}^r$, the action of $\Gal_K$ on $M\otimes_{\theta\circ \varphi^{-n}} \C_p$ actually factor through the an action of $\Gal(L'/K)$ on $M\otimes_{\theta\circ \varphi^{-n}} K'_n$, and there is a basis of $M\otimes_{\theta\circ \varphi^{-n}} K'_n$ fixed by $\Gal(L'/L)$. 

Now we claim that under the notions as above, we can show that there is an open subgroup of $\Gal(L'/K)$ fixes a basis of $M\otimes_{\theta\circ \varphi^{-n}} K'_n$, and this will imply that $T$ is $\C_p$-admissible, thus potentially unramified by a theorem of Sen \cite[\S3.2]{Sen-Cont}. To show this claim, we have $\Gal(L'/K'_n)$ acts linearly on $M\otimes_{\theta\circ \varphi^{-n}} K'_n$ and $H=\Gal(L'/K'_n)\cap\Gal(L'/L)$ in the kernel of this action, it remains to show that the normalization of $H$ in $\Gal(L'/K'_n)$ is $\Gal(L'/K'_n)$. Then this is actually from Kummer theory since $K'_n$ contains both $p^n$-th roots of unit and a $p^n$-th root of $p$. The above can also be proven by a direct computation by identifying $\Gal(L'/K'_n)$ with $\Z_p\rtimes \Z_p$ when $n\gg 0$.

\bibliographystyle{amsalpha}
\bibliography{mybib}

\end{document}